\documentclass[11pt]{article}
\usepackage{color}
\usepackage[margin=1in]{geometry}
\usepackage{amssymb}
\usepackage{amsmath}
\usepackage{amsthm}
\usepackage{graphicx}

\newtheorem{theorem}{Theorem}
\newtheorem{lemma}{Lemma}
\newtheorem{definition}{Definition}

\newtheorem{corollary}{Corollary}
\newtheorem{remark}{Remark}
\newcommand{\R}{\mathbb{R}}
\newcommand{\T}{\mathcal{T}}
\newcommand{\norm}[1]{\left\lVert#1\right\rVert}

\title{Two-hidden-layer ReLU neural networks and finite elements}
\author{Pengzhan Jin\footnote{National Engineering Laboratory for Big Data Analysis and Applications, Peking University, Beijing 100871, China (jpz@pku.edu.cn).}}
\date{}

\begin{document}

\maketitle

\begin{abstract}
We point out that (continuous or discontinuous) piecewise linear functions on a convex polytope mesh can be represented by two-hidden-layer ReLU neural networks in a weak sense. In addition, the numbers of neurons of the two hidden layers required to weakly represent are accurately given based on the numbers of polytopes and hyperplanes involved in this mesh. The results naturally hold for constant and linear finite element functions. Such weak representation establishes a bridge between two-hidden-layer ReLU neural networks and finite element functions, and leads to a perspective for analyzing approximation capability of ReLU neural networks in $L^p$ norm via finite element functions. Moreover, we discuss the strict representation for tensor finite element functions via the recent tensor neural networks.
\end{abstract}

\section{Introduction}
The properties of neural networks (NNs) have been widely studied. The NNs are proved to be universal approximators in \cite{cybenko1989approximation,hornik1989multilayer}, and shown powerful expressive capability in \cite{e2019barron,siegel2022high}. Especially, the NNs with ReLU \cite{nair2010rectified} activation receive special attention \cite{daubechies2022nonlinear,li2017convergence,lu2019dying,yarotsky2017error}. It is well known that every ReLU NN represents a continuous piecewise linear function \cite{pascanu2013number}. Conversely, whether any continuous piecewise linear function can be expressed by a ReLU NN is an interesting problem. \cite{arora2016understanding} shows that any continuous piecewise linear function can be represented by a ReLU NN with at most $\lceil\log_2(n + 1)\rceil$ hidden layers where $n$ is the dimension. Some estimations of the number of neurons needed for such representation are presented in \cite{he2018relu}. The recent works \cite{chen2022improved,hertrich2021towards} further derive better bounds for representing piecewise linear functions. Besides the discussion on classical ReLU NNs, the representations of continuous piecewise linear
functions with infinite width shallow ReLU NNs defined in an integral form are analyzed in \cite{mccarty2023piecewise}. Moreover, \cite{he2023deep} constructs continuous piecewise high-order polynomial functions by using ReLU and ${\rm ReLU}^2$ activation functions.

The finite element method \cite{brenner2008mathematical,ciarlet2002finite} is a powerful computational technique widely used for solving complex engineering and physical problems governed by differential equations. The aforementioned continuous piecewise linear or higher-order polynomial functions form a principal research subject in finite element methods, with the field's characteristic emphasis on particular domain subdivisions, notably simplex decompositions. As the finite element method has been well developed, we expect to establish a more precise connection between NNs and finite element functions, e.g. constant and linear finite element functions that are piecewise constant on general polytope meshes and piecewise linear on simplex meshes, respectively. Different from the previously mentioned works, here we focus on a weak representation which is equivalent to the strict representation in $L^p$ norm. Note that the strict representation given by \cite{arora2016understanding,chen2022improved,he2018relu,hertrich2021towards} requires deep NNs, while the weak representation can be established with only two hidden layers. We provide the accurate number of neurons needed for such weak representation based on the numbers of polytopes and hyperplanes involved in the mesh, which means the network size needed for weak representation is in fact computable. Another topic in this work is the connection between the tensor finite element functions and the recently proposed tensor neural networks \cite{jin2022mionet}, which are applied to solving high-dimensional eigenvalue problems \cite{wang2022tensor}, based on the fact that the rank-one tensor decomposition transforms high-dimensional integrals into products of low-dimensional integrals. In such a case, a strict representation can be obtained. We additionally show several specific examples and demonstrate how to compute the numbers of neurons for representation given meshes, for constant, linear and tensor finite element functions.

The paper is organized as follows. We first clarify the network architectures to be studied in Section \ref{sec:net}. In Section \ref{sec:constant_linear}, we give the concept of weak representation and prove the main theorem for weak representation of two-hidden-layer ReLU NNs as well as several related corollaries. Subsequently, we show the strict representation of tensor neural networks for tensor finite element functions in Section \ref{sec:tensor}. Finally, Section \ref{sec:conclusions} summarizes this work.

\section{ReLU neural networks}\label{sec:net}
We consider two types of neural networks. The first type is ReLU fully-connected neural networks with two hidden layers which are defined as
\begin{equation}
    f_{NN}(x):=w_3\sigma(W_2\sigma(W_1x+b_1)+b_2),\quad x\in\R^n,
\end{equation}
where $W_1\in\R^{h_1\times n},b_1\in\R^{h_1},W_2\in\R^{h_2\times h_1},b_2\in\R^{h_2},w_3\in\R^{1\times h_2}$, and $\sigma(x):=\max(x,0)$ is the ReLU activation function considered as element-wise mapping for varying dimension. Here $h_1,h_2$ are the numbers of neurons of the two hidden layers, and we denote the set of such neural networks as ${\rm FNN}(h_1,h_2)$.

The second type is the neural network with a tensor representation, which is firstly introduced in MIONet \cite{jin2022mionet} for operator regression. Tensor neural network (TNN) with one hidden layer is defined as
\begin{equation}
    t_{NN}(x_1,\cdots,x_k)=\sum_{j=1}^p(w_1^j\sigma(W_1x_1+b_1))\cdots(w_k^j\sigma(W_kx_k+b_k)),\quad x_i\in\R^{n_i},
\end{equation}
where $n_1+\cdots+n_k=n$, $W_i\in\R^{h_i\times n_i}$, $b_i\in\R^{h_i}$, $w_i^j\in\R^{1\times h_i}$, and $\sigma$ is the activation. Here $p$ is regarded as the rank of this network. In fact, $t_{NN}$ is a tensor which can be written as
\begin{equation}
    t_{NN}=\sum_{j=1}^pl_1^j\otimes\cdots\otimes l_k^j,\quad l_i^j(x):=w_i^j\sigma(W_ix+b_i).
\end{equation}
Especially, here we mainly discuss the TNNs with $n_i=1,k=n$, and the ReLU activation. Denote the set of  such tensor neural networks as ${\rm TNN}^p(h_1,\cdots,h_n)$.

The above two described architectures are both the mappings from $\R^n$ to $\R$, and the TNNs also satisfy the approximation theorem according to \cite{jin2022mionet,wang2022tensor}. 

\section{Constant and linear finite elements}\label{sec:constant_linear}
Here we first consider the constant and linear finite elements, and search for weak representation for such elements via ReLU
neural networks. In this context, constant finite element functions correspond to piecewise constant functions on polytope meshes, while linear finite element functions correspond to continuous piecewise linear functions on simplex meshes. Precise definitions will be given below.

In this section we study the closed polytope $\Omega$ (unnecessary to be convex) in Euclidean space $\R^n$. Assume that $\T$ is a polytope mesh of $\Omega$ whose elements are convex, precisely,
\begin{equation}
    \T=\{\tau_1,\cdots,\tau_{N_{\T}}\},
\end{equation}
where $\tau_i\subset\Omega$ is $n$-dimensional closed convex polytope, $\bigcup_{i=1}^{N_\T}\tau_i=\Omega$, and $\mathring{\tau_i}\cap\mathring{\tau_j}=\varnothing$ for all $i\neq j$. We denote the numbers of interior and boundary hyperplanes of $\T$ by $H_\T^i$ and $H_\T^b$ respectively, then the number of all hyperplanes of $\T$ is $H_\T=H_\T^i+H_\T^b$.

Piecewise linear function on mesh $\T$ is defined as
\begin{equation}
    f(x):=a_{\tau}x+c_{\tau}\quad{\rm for}\quad x\in\mathring{\tau},\tau\in\T,
\end{equation}
where $a_\tau\in\R^{1\times n}$ and $c_\tau\in\R$ are related to $\tau$. Note that we do not care about the values of $f$ on $\partial\tau$. Denote the set consisting of all the piecewise linear functions on mesh $\T$ by $\mathcal{V}_\T$. It is obvious that $\mathcal{V}_\T$ contains all the constant finite element functions on $\T$, denoted by $\mathcal{V}_\T^0$, i.e.
\begin{equation}
    \mathcal{V}_\T^0:=\{f\in\mathcal{V}_\T:f|_{\mathring{\tau}}{\rm\ is\ constant\ for\ any\ }\tau\in\mathcal{T}\}.
\end{equation}
In addition, $\mathcal{V}_\T$ will contain all the linear finite element functions on $\T$ denoted by $\mathcal{V}_\T^1$ if $\T$ is given as a simplex mesh for finite element methods, i.e.
\begin{equation}
    \mathcal{V}_\T^1:=\{f\in\mathcal{V}_\T:f{\rm\ is\ continuous\ on\ }\Omega\}{\rm\ when\ \T\ is\ a\ simplex\ mesh}.
\end{equation}
For small $\epsilon>0$, denote
\begin{equation}
    \Omega_\T^\epsilon:=\bigcup_{\tau\in\T}\{x\in\tau:d(x,\partial\tau)\geq\epsilon\},
\end{equation}
where $d(x,\partial\tau)$ is the distance between $x$ and $\partial\tau$.

\subsection{Weak representation}
We begin with the definition of weak representation.
\begin{definition}
Assume that $\T$ is a convex polytope mesh of $\Omega$, $\mathcal{V}$ is a set of functions on $\Omega$. We say a function set $\mathcal{F}$ weakly represents $\mathcal{V}$ on $\T$, if for any $v\in\mathcal{V}$ and any $\epsilon>0$, there exists a $f\in\mathcal{F}$ such that $f|_{\Omega_\T^\epsilon}=v|_{\Omega_\T^\epsilon}$ and $\norm{f}_{L^\infty(\Omega)}\leq \norm{v}_{L^\infty(\Omega)}$.%the set $\{f\in\mathcal{F}:f|_{\Omega_\T^\epsilon}\equiv v|_{\Omega_\T^\epsilon},\norm{f}_{L^\infty(\Omega)}\leq \norm{v}_{L^\infty(\Omega)}\}$ is nonempty.}
\end{definition}
The main theorem is given as follows.
\begin{theorem}\label{thm:cl_weak_rep}
${\rm FNN}(2H_\T^i+H_\T^b,N_\T+1)$ weakly represents $\mathcal{V}_\T$ on $\T$.
\end{theorem}
\noindent Recall that the constant finite element functions on $\T$ are contained in $\mathcal{V}_\T$, while the linear finite element functions on $\T$ will be contained in $\mathcal{V}_\T$ given $\T$ is further a simplex mesh for finite element methods.
\begin{corollary}\label{cor:cfe_weak_rep}
${\rm FNN}(2H_\T^i+H_\T^b,N_\T+1)$ weakly represents the set of constant finite element functions $\mathcal{V}_\T^0$ on $\T$.
\end{corollary}
\begin{corollary}\label{cor:lfe_weak_rep}
${\rm FNN}(2H_\T^i+H_\T^b,N_\T+1)$ weakly represents the set of linear finite element functions $\mathcal{V}_\T^1$ on $\T$ when $\T$ is a simplex mesh for finite element methods.
\end{corollary}
\begin{corollary}\label{cor:estimate}
$\overline{{\rm FNN}(2H_\T^i+H_\T^b,N_\T+1)}_{L^p(\Omega
)}\supset \mathcal{V}_\T$ for $1\leq p<\infty$.
\end{corollary}
Corollary \ref{cor:estimate} shows that the constant/linear finite element functions on $\T$ can be approximated by two-hidden-layer neural networks with neurons $(2H_\T^i+H_\T^b,N_\T+1)$ in arbitrary accuracy under $L^p$ norm. Note that Theorem \ref{thm:cl_weak_rep} and Corollary \ref{cor:cfe_weak_rep}-\ref{cor:estimate} also hold for ${\rm FNN}(2H_\T,N_\T+1)$ due to $2H_\T^i+H_\T^b<2H_\T$. 

The concept of weak representation and corresponding theorem and corollaries are established mainly for two reasons. The first reason is that we expect to connect finite element functions with ReLU NNs whose number of hidden layers has a constant upper bound, while a strict representation requires $\lceil\log_2(n + 1)\rceil$ hidden layers depending on the dimension $n$, based on current progress. The second reason is that ReLU NNs are unable to strictly represent discontinuous piecewise linear/constant functions, especially the constant finite element functions which are significant for finite element methods. Therefore a weaker setting is needed for establishing the connection. In the following we adopt a constructive proof, where we construct an appropriate compactly supported two-hidden-layer NN on each polytope unit and sum all the basis functions via the output layer. Such construction naturally motivates the above definition of weak representation.

We next show the proof in detail.

\begin{lemma}\label{lem:positive}
Let $K=\{x\in \mathbb{R}^{n}:h_{i}(x):=w_{i}x+b_{i}\geq0,1\leq i\leq m\}$ be a convex $m$-polytope ($m\geq n+1$) in $\mathbb{R}^{n}$, then there exist $\lambda_{i}>0$, $1\leq i\leq m$, such that
\begin{equation}
\sum_{i=1}^m\lambda_{i}w_i=0.
\end{equation}
\end{lemma}

\begin{proof}
This lemma is a fundamental result in polyhedral theory, and readers familiar with it may skip the following elementary proof.

Choose a fixed $y\in \mathring{K}$, and denote by $y_i$ the projection points of $y$ on the hyperplanes $h_i(x)=0$, then $y-y_{i}=a_{i}w_{i}^{T}$, $a_{i}>0$. We assert that there exist $i\neq j$ such that $h(y_{i})\cdot h(y_{j})<0$, given any hyperplane $h(x)=wx-wy=0$ passing through $y$. Otherwise without loss of generality we assume that $h(y_{i})\geq0$ for all $1\leq i\leq m$, then $h(y_{i})=w(y_{i}-y)=-a_{i}ww_{i}^{T}\geq0$, therefore $h_{i}(y-tw^{T})=w_{i}y+b_{i}-tw_{i}w^{T}\geq 0$ holds for all $1\leq i\leq m$ and $t>0$. As a result $y-tw^{T}\in K$ for all $t>0$, this is contradictory with the fact that $K$ is bounded. This assertion points out that $y$ is contained in the convex hull $G$ generated by $\{y_1,...,y_m\}$, furthermore $y\in \mathring{G}$.

Assume that $y_1,...,y_{k}$ are the $k$ vertices of $G$, $n+1\leq k\leq m$. We first choose $\eta_{k+1},...,\eta_{m}>0$ small enough such that
\begin{equation}
    y':=y-\eta_{k+1}(y_{k+1}-y)-\cdots-\eta_{m}(y_m-y)\in \mathring{G}.
\end{equation}
As $G$ is a convex polytope and $y'\in \mathring{G}$, we can find $0<\eta_1,...,\eta_k<1$ satisfying $\eta_1+\cdots+\eta_k=1$ such that
\begin{equation}
y'=\eta_1y_1+\eta_2y_2+\cdots+\eta_ky_k,
\end{equation}
where $\eta_1,...,\eta_k$ are in fact the positive barycentric coordinates and the existence is given by \cite{lee1990some} and mentioned in \cite{floater2015generalized}. Subsequently, we have
\begin{equation}\label{eq:barycentric}
y=\eta_1y_1+\eta_2y_2+\cdots+\eta_ky_k+\eta_{k+1}(y_{k+1}-y)+\cdots+\eta_m(y_m-y),
\end{equation}
where
\begin{equation}
0<\eta_1,\eta_2,...,\eta_m<1,\quad\eta_1+\eta_2+\cdots+\eta_k=1.
\end{equation}
Eq. \eqref{eq:barycentric} is equivalent to
\begin{equation}
\eta_1a_1w_1+\eta_2a_2w_2+\cdots+\eta_ma_mw_m=0,
\end{equation}
thus $\lambda_i:=\eta_ia_i>0$ ($1\leq i\leq m$) are what we need.
\end{proof}
After Lemma \ref{lem:positive}, the proof of Theorem \ref{thm:cl_weak_rep} can be given. In this paper, where the 2-norm of vectors appears frequently, we simplify the notation by replacing $\norm{\cdot}_2$ with $|\cdot|$, provided that no ambiguity arises.

\textbf{Proof of Theorem \ref{thm:cl_weak_rep}.}
Assume that $v\in\mathcal{V}_\T$ is a piecewise linear function, $\tau\in\T$ is a convex polytope, and denote
\begin{equation}
\tau^\epsilon:=\{x\in\tau:d(x,\partial\tau)\geq\epsilon\}.
\end{equation}
Since the convex polytope is surrounded by several hyperplanes, we denote
\begin{equation} 
\begin{cases}
\tau=\{x\in\R^n:h_i(x):=w_ix+b_i\geq 0,1\leq i\leq m\}\\
\tau^\epsilon=\{x\in\R^n:h_i^\epsilon(x):=w_i x+b_i-\epsilon|w_i|\geq 0,1\leq i\leq m\}
\end{cases},
\quad m\geq n+1,
\end{equation}
Assume that $v|_{\mathring{\tau}}(x)=ax+c$, $a\in\R^{1\times n}$, $c\in\R$. Consider the linear system
\begin{equation}\label{eq:linear_system}
	(w_1^T,w_2^T,...,w_m^T)\cdot \mu = -a^T,\quad \mu=(\mu_1,\mu_2,...,\mu_m)^T\in\R^m.
\end{equation}
Since $\tau$ is an $n$-dimensional convex polytope, $(w_1^T,w_2^T,...,w_m^T)$ is a full rank $n$-by-$m$ matrix, hence there exists a solution $\mu$ for system \eqref{eq:linear_system}. By Lemma \ref{lem:positive} we know there exist $\lambda_1,...,\lambda_m>0$ such that the vector
\begin{equation}
(\mu_1+t\lambda_1,\mu_2+t\lambda_2,...,\mu_m+t\lambda_m)^T
\end{equation}
solves system \eqref{eq:linear_system} for all $t\in\R$. In the following discussion we consider the $t$ large enough so that $\mu_i+t\lambda_i>0$ for $1\leq i\leq m$. Denote 
\begin{equation}
\begin{cases}
W_{\rm I}^{\tau}:=(w_1^T,w_2^T,...,w_m^T)^T\in\R^{m\times n},\\
b_{\rm I}^{\tau}:=(b_1-\epsilon|w_1|,b_2-\epsilon|w_2|,...,b_m-\epsilon|w_m|)^T\in\R^m,
\end{cases}
\end{equation}
and
\begin{equation}
\begin{cases}
\hat{w}_{\rm II}^\tau(t):=-(\mu_1+t\lambda_1,\mu_2+t\lambda_2,...,\mu_m+t\lambda_m)\in\R^{1\times m},\\
\hat{b}_{\rm II}^\tau(t):=\sum_{i=1}^m(\mu_i+t\lambda_i)(b_i-\epsilon|w_i|)+c+R\in\R,
\end{cases}
\end{equation}
where $R:=\norm{v}_{L^\infty(\Omega)}$ . Now consider
\begin{equation}
f_t^\tau(x):=\sigma(\hat{w}_{\rm II}^\tau(t)\sigma(W_{\rm I}^{\tau}x+b_{\rm I}^{\tau})+\hat{b}_{\rm II}^\tau(t)),
\end{equation}
and we investigate the properties of $f_t^\tau$. Firstly, we can check that
\begin{equation}
\begin{split}
f_t^\tau|_{\tau^\epsilon}(x)=&\sigma(\hat{w}_{\rm II}^\tau(t)(W_{\rm I}^{\tau}x+b_{\rm I}^{\tau})+\hat{b}_{\rm II}^\tau(t)) \\
=&\sigma(ax+c+R) \\
=&v|_{\tau^\epsilon}(x)+R,\quad \forall x\in\tau^\epsilon.
\end{split}
\end{equation}
For $x\in\tau\backslash\tau^\epsilon$, we have
\begin{equation}
\hat{w}_{\rm II}^\tau(t)\sigma(W_{\rm I}^{\tau}x+b_{\rm I}^{\tau})+\hat{b}_{\rm II}^\tau(t)\leq\hat{w}_{\rm II}^\tau(t)(W_{\rm I}^{\tau}x+b_{\rm I}^{\tau})+\hat{b}_{\rm II}^\tau(t)=ax+c+R\leq 2R,
\end{equation}
due to $\mu_i+t\lambda_i>0$ ($1\leq i\leq m$) and $\sigma(y)\geq y\ \forall y\in\R$. Therefore
\begin{equation}
0\leq f_t^\tau(x)=\sigma(\hat{w}_{\rm II}^\tau(t)\sigma(W_{\rm I}^{\tau}x+b_{\rm I}^{\tau})+\hat{b}_{\rm II}^\tau(t))\leq 2R,\quad \forall x\in\tau\backslash\tau^\epsilon.
\end{equation}
Moreover, if $x\in\R^n\backslash\tau$, we temporarily fix $x$ and consider the index sets
\begin{equation}
\begin{cases}
    I_{+}:=\{1\leq i\leq m:h_i(x)=w_ix+b_i\geq 0\},\\
    I_{-}:=\{1\leq i\leq m:h_i(x)=w_ix+b_i< 0\},
\end{cases}
\end{equation}
then $I_{-}$ is nonempty, and
\begin{equation}\label{eq:prf_1}
\begin{split}
&\hat{w}_{\rm II}^\tau(t)\sigma(W_{\rm I}^{\tau}x+b_{\rm I}^{\tau})+\hat{b}_{\rm II}^\tau(t)\\=&-\sum_{i=1}^m(\mu_i+t\lambda_i)\sigma(w_ix+b_i-\epsilon|w_i|)+\sum_{i=1}^m(\mu_i+t\lambda_i)(b_i-\epsilon|w_i|)+c+R \\
=&-\sum_{i\in I_{+}}(\mu_i+t\lambda_i)\sigma(w_ix+b_i-\epsilon|w_i|)-\sum_{i\in I_{-}}(\mu_i+t\lambda_i)\sigma(w_ix+b_i-\epsilon|w_i|)\\&+\sum_{i=1}^m(\mu_i+t\lambda_i)(b_i-\epsilon|w_i|)+c+R\\
=&-\sum_{i\in I_{+}}(\mu_i+t\lambda_i)\sigma(w_ix+b_i-\epsilon|w_i|)+\sum_{i=1}^m(\mu_i+t\lambda_i)(b_i-\epsilon|w_i|)+c+R,
\end{split}
\end{equation}
since $\sigma(w_ix+b_i-\epsilon|w_i|)=0$ for $i\in I_{-}$. Next, we derive that
\begin{equation}\label{eq:prf_2}
\begin{split}
    &-\sum_{i\in I_{+}}(\mu_i+t\lambda_i)\sigma(w_ix+b_i-\epsilon|w_i|) \\
    \leq&-\sum_{i\in I_{+}}(\mu_i+t\lambda_i)(w_ix+b_i-\epsilon|w_i|) \\
    =&\sum_{i\in I_{-}}(\mu_i+t\lambda_i)(w_ix+b_i-\epsilon|w_i|)-\sum_{i=1}^m(\mu_i+t\lambda_i)(w_ix+b_i-\epsilon|w_i|).
\end{split}
\end{equation}
By combining \eqref{eq:prf_1} and \eqref{eq:prf_2}, we have
\begin{equation}\label{eq:prf_3}
\begin{split}
&\hat{w}_{\rm II}^\tau(t)\sigma(W_{\rm I}^{\tau}x+b_{\rm I}^{\tau})+\hat{b}_{\rm II}^\tau(t) \\
\leq&\sum_{i\in I_{-}}(\mu_i+t\lambda_i)(w_ix+b_i-\epsilon|w_i|)-\sum_{i=1}^m(\mu_i+t\lambda_i)(w_ix+b_i-\epsilon|w_i|)\\&+\sum_{i=1}^m(\mu_i+t\lambda_i)(b_i-\epsilon|w_i|)+c+R\\
=&\sum_{i\in I_{-}}(\mu_i+t\lambda_i)(w_ix+b_i-\epsilon|w_i|)-\sum_{i=1}^m(\mu_i+t\lambda_i)(w_ix)+c+R.
\end{split}
\end{equation}
We expect to find a suitable $t$ such that $\hat{w}_{\rm II}^\tau(t)\sigma(W_{\rm I}^{\tau}x+b_{\rm I}^{\tau})+\hat{b}_{\rm II}^\tau(t)\leq 0$. To this end, we first introduce a constant
\begin{equation}
s:=\max_{1\leq i\leq m}\left|\frac{\mu_i}{\lambda_i}\right|+1,
\end{equation}
which makes $\mu_i+s\lambda_i$ positive for all $1\leq i\leq m$, then we can rewrite \eqref{eq:prf_3} as
\begin{equation}
\begin{split}
&\hat{w}_{\rm II}^\tau(t)\sigma(W_{\rm I}^{\tau}x+b_{\rm I}^{\tau})+\hat{b}_{\rm II}^\tau(t) \\
\leq&\sum_{i\in I_{-}}(\mu_i+t\lambda_i)(w_ix+b_i-\epsilon|w_i|)-\sum_{i=1}^m(\mu_i+t\lambda_i)(w_ix)+c+R \\
=&\sum_{i\in I_{-}}(\mu_i+t\lambda_i)(w_ix+b_i-\epsilon|w_i|)-\sum_{i=1}^m(\mu_i+s\lambda_i)(w_ix+b_i-b_i)+c+R\\
=&\sum_{i\in I_{-}}(\mu_i+t\lambda_i)(w_ix+b_i-\epsilon|w_i|)-\sum_{i=1}^m(\mu_i+s\lambda_i)(w_ix+b_i)\\
&+\sum_{i=1}^m(\mu_i+s\lambda_i)b_i+c+R\\
=&-\sum_{i\in I_{-}}(\mu_i+t\lambda_i)\epsilon|w_i|+\sum_{i\in I_{-}}(t-s)\lambda_i(w_ix+b_i)-\sum_{i\in I_{+}}(\mu_i+s\lambda_i)(w_ix+b_i)\\
&+\sum_{i=1}^m(\mu_i+s\lambda_i)b_i+c+R,
\end{split}
\end{equation}
where the first equality is due to $\sum_{i=1}^m(\mu_i+t\lambda_i)w_i=\sum_{i=1}^m(\mu_i+s\lambda_i)w_i=-a$ according to the definitions of $\mu_i$ and $\lambda_i$. It is sufficient to find a $t$ such that:
\begin{itemize}
    \item $-\sum_{i\in I_{+}}(\mu_i+s\lambda_i)(w_ix+b_i)\leq 0$.
    \item $+\sum_{i\in I_{-}}(t-s)\lambda_i(w_ix+b_i)\leq 0$.
    \item $-\sum_{i\in I_{-}}(\mu_i+t\lambda_i)\epsilon|w_i|+\sum_{i=1}^m(\mu_i+s\lambda_i)b_i+c+R\leq 0$.
\end{itemize}
Note that the first inequality naturally holds since $\mu_i+s\lambda_i\geq0$ and $w_ix+b_i\geq 0$ for $i\in I_{+}$. The second and third inequalities are in fact both affine functions on $t$, with the linear coefficients $\sum_{i\in I_{-}}\lambda_i(w_ix+b_i)$ and $-\sum_{i\in I_{-}}\lambda_i\epsilon|w_i|$, which are both negative, as $I_{-}$ is a nonempty set. Consequently, we can always find a sufficiently large $t$ such that the above inequalities holds, for example,
\begin{equation}
t_0:=\max\left(\frac{|\sum_{i=1}^m(\mu_i+s\lambda_i)b_i+c+R|+\sum_{i=1}^m\epsilon|\mu_i||w_i|}{\min_{1\leq i\leq m}\epsilon\lambda_i|w_i|}, s+1\right).
\end{equation}
Now we have $\hat{w}_{\rm II}^\tau(t_0)\sigma(W_{\rm I}^{\tau}x+b_{\rm I}^{\tau})+\hat{b}_{\rm II}^\tau(t_0)\leq 0$ for the fixed $x\in\R^n\backslash\tau$. Furthermore, it is noticed that the definition of $t_0$ is independent of $x$, thus we can assert that
\begin{equation}
\hat{w}_{\rm II}^\tau(t_0)\sigma(W_{\rm I}^{\tau}x+b_{\rm I}^{\tau})+\hat{b}_{\rm II}^\tau(t_0)\leq 0,\quad \forall x\in\R^n\backslash\tau,
\end{equation}
subsequently we derive that
\begin{equation}
f_{t_0}^\tau(x)=\sigma(\hat{w}_{\rm II}^\tau(t_0)\sigma(W_{\rm I}^{\tau}x+b_{\rm I}^{\tau})+\hat{b}_{\rm II}^\tau(t_0))=0,\quad \forall x\in\R^n\backslash\tau.
\end{equation}
We simplify the notations by
\begin{equation}\label{eq:phi_tau}
\phi^\tau(x):=\sigma(w_{\rm II}^\tau\sigma(W_{\rm I}^{\tau}x+b_{\rm I}^{\tau})+b_{\rm II}^\tau),\quad w_{\rm II}^\tau:=\hat{w}_{\rm II}^\tau(t_0),\ b_{\rm II}^\tau:=\hat{b}_{\rm II}^\tau(t_0),
\end{equation}
and summarize the above results as
\begin{equation}\label{eq:phi_tau_pro}
\begin{cases}
\phi^\tau(x)=v|_{\tau^\epsilon}(x)+R,\quad &x\in\tau^\epsilon,\\
0\leq \phi^\tau(x)\leq 2R,\quad&x\in\tau\backslash\tau^\epsilon,\\
\phi^\tau(x)=0,\quad &x\in\R^n\backslash\tau.
\end{cases}
\end{equation}

Finally, we are able to construct the expected FNN. Denote
\begin{equation}\label{eq:W1B1}
W_{\rm I}:=\begin{pmatrix}
W_{\rm I}^{\tau_1} \\
W_{\rm I}^{\tau_2} \\
\vdots \\
W_{\rm I}^{\tau_{N_{\mathcal{T}}}} \\
\end{pmatrix},\quad B_{\rm I}:=
\begin{pmatrix}
b_{\rm I}^{\tau_1} \\
b_{\rm I}^{\tau_2} \\
\vdots \\
b_{\rm I}^{\tau_{N_{\mathcal{T}}}} \\
\end{pmatrix},
\end{equation}
as well as
\begin{equation}\label{eq:W2B2}
 W_{\rm II}:={\rm diag}\left(w_{\rm II}^{\tau_1},w_{\rm II}^{\tau_2},...,w_{\rm II}^{\tau_{N_{\mathcal{T}}}},O\right),\quad B_{\rm II}:=\begin{pmatrix}
b_{\rm II}^{\tau_1} \\
b_{\rm II}^{\tau_2} \\
\vdots \\
b_{\rm II}^{\tau_{N_{\mathcal{T}}}} \\
R
\end{pmatrix},\quad w_{\rm III}:=\left(1,1,...,1,-1\right)_{1\times (N_{\mathcal{T}}+1)},
\end{equation}
where $O:=(0,0,...,0)$ whose dimension equals the number of rows of $W_{\rm I}^{\tau_{N_{\mathcal{T}}}}$, and ${\rm diag}(\cdot,...,\cdot)$ means the block diagonal matrix with given diagonal blocks. It is readily to check that
\begin{equation}\label{eq:fnn_constructed}
f(x):=w_{\rm III}\sigma(W_{\rm II}\sigma(W_{\rm I}x+B_{\rm I})+B_{\rm II})=-R+\sum_{i=1}^{N_{\mathcal{T}}}\phi^{\tau_i}(x)
\end{equation}
satisfies:
\begin{itemize}
\item $f|_{\Omega_\T^\epsilon}=v|_{\Omega_\T^\epsilon}$.
\item $\norm{f}_{L^\infty(\Omega)}\leq R=\norm{v}_{L^\infty(\Omega)}$.
\end{itemize}
The last issue is the neurons of $f$. The number of neurons of the second hidden layer is indeed $N_{\mathcal{T}}+1$. As for the first hidden layer, it is noticed that $W_{\rm I}^{\tau_i},b_{\rm I}^{\tau_i}$ are determined by the directed hyperplanes of $\T$. We can remove the duplicate directed hyperplanes and subsequently reduce the number of neurons to exactly $2H_\T^i+H_\T^b$, related to the number of all the directed hyperplanes involved in the mesh. Specifically, assume that there are two basis functions $\sigma(wx+b-\epsilon|w|)$ and $\sigma(\tilde{w}x+\tilde{b}-\epsilon|\tilde{w}|)$ related to a same directed hyperplane, which means
\begin{equation}
    \{x\in\R^n:wx+b\geq 0\}=\{x\in\R^n:\tilde{w}x+\tilde{b}\geq 0\},
\end{equation}
so that
\begin{equation}
    \tilde{w}=\lambda w,\quad \tilde{b}=\lambda b,
\end{equation}
for a constant $\lambda>0$. Then
\begin{equation}
    \sigma(\tilde{w}x+\tilde{b}-\epsilon|\tilde{w}|)=\sigma(\lambda(wx+b-\epsilon|w|))=\lambda\sigma(wx+b-\epsilon|w|),
\end{equation}
they are linear dependent and can be merged into one.
\hfill $\square$

\begin{remark}
Based on the proof, it is noticed that if we consider the two-hidden-layer FNNs with a bias at the output layer, then the number of neurons of the second hidden layer will be reduced by 1. In such a case, we can replace the ${\rm FNN}(2H_\T^i+H_\T^b,N_\T+1)$ in Theorem \ref{thm:cl_weak_rep} and Corollary \ref{cor:cfe_weak_rep}-\ref{cor:estimate} by ${\rm FNN_b}(2H_\T^i+H_\T^b,N_\T)$.
\end{remark}
\begin{remark}\label{rem:complexity}
Following the construction procedure, the worst-case computational complexity for constructing the weakly represented neural network is $\mathcal{O}(n^3MN_\T)+{\rm C}(n,M)\cdot N_\T$, where $M$ is the maximum number of facets among the polytopes ($M=n+1$ for simplex mesh), and ${\rm C}(n,M)$ is the complexity for computing $\lambda_1,...,\lambda_m$ in Lemma \ref{lem:positive} for $m=M$, which may involve a linear programming algorithm.
\end{remark}
\begin{remark}
Two hidden layers are necessary for weak representation when $n\geq 2$. Consider the set of all the non-differentiable points of a nonzero one-hidden-layer ReLU NN, it is indeed the union of finitely many hyperplanes. However, the set of non-differentiable points of a ReLU NN weakly representing a linear finite element function on simplex mesh with small $\epsilon$ does not equal the union of finitely many hyperplanes (restricted to $\Omega$), as long as the union of boundaries of the simplexes does not equal the union of a finite number of hyperplanes (restricted to $\Omega$). Such mesh exists when $n\geq 2$.
\end{remark}
\begin{remark}
In the previous literature \cite{arora2016understanding,he2018relu}, it has been established that continuous piecewise linear functions can be strictly represented by ReLU NNs with at most $\lceil\log_2(n + 1)\rceil$ hidden layers, and an estimate of the number of neurons of $\mathcal{O}(n2^{N_\T N_\T!})$ is further provided. \cite{chen2022improved} improves the number of neurons to $\mathcal{O}(N_\T^2)$ with $2\lceil\log_2(N_\T)\rceil+1$ hidden layers. In comparison, we show that (continuous or discontinuous) piecewise linear functions can be weakly represented by ReLU NNs with two hidden layers, and the number of neurons is accurately given by $2H_\T^i+H_\T^b + N_\T$. Note that $2H_\T^i+H_\T^b\leq M N_\T$ where $M$ is defined in Remark \ref{rem:complexity} ($M=n+1$ for simplex mesh).
\end{remark}
Through careful modification of technical details in the preceding theorem's demonstration, we derive the subsequent corollary.
\begin{corollary}
Any function in $L^p(\R^n)$ ($1\leq p<\infty$) can be arbitrarily well-approximated in the $L^p(\R^n)$ norm by a two-hidden-layer ReLU neural network. Furthermore, two hidden layers are necessary when $n\geq 2$.
\end{corollary}
\begin{proof}
As the family of compactly supported linear finite element functions are dense in $L^p(\R^n)$, we only need to show that the NN constructed in the proof of Theorem \ref{thm:cl_weak_rep} can be modified so that it is compactly supported in a convex $\Omega$. Recall that the NN in Eq. (\ref{eq:fnn_constructed}) satisfies:
\begin{itemize}
\item $f|_{\Omega_\T^\epsilon}=v|_{\Omega_\T^\epsilon}$.
\item $\norm{f}_{L^\infty(\Omega)}\leq R=\norm{v}_{L^\infty(\Omega)}$.
\item $f|_{\R^n\backslash\Omega}=-R$.
\end{itemize}
It suffices to adjust $f|_{\R^n\backslash\Omega}$ to zero. Firstly, we construct a function
\begin{equation}
\phi^\Omega(x):=\sigma(w_{\rm II}^\Omega\sigma(W_{\rm I}^{\Omega}x+b_{\rm I}^{\Omega})+b_{\rm II}^\Omega),
\end{equation}
which satisfies
\begin{equation}
\begin{cases}
\phi^\Omega(x)=R,\quad &x\in\Omega^\epsilon,\\
0\leq \phi^\Omega(x)\leq R,\quad&x\in\Omega\backslash\Omega^\epsilon,\\
\phi^\Omega(x)=0,\quad &x\in\R^n\backslash\Omega,
\end{cases}
\end{equation}
in the same manner as Eq. (\ref{eq:phi_tau}) and Eq. (\ref{eq:phi_tau_pro}), by considering the convex $\Omega$ and the constant function $\tilde{v}(x)\equiv R/2$ instead of $\tau$ and $v$. Secondly, we augment $W_{\rm I}$ and $B_{\rm I}$ in \eqref{eq:W1B1} with $W_{\rm I}^\Omega$ and $b_{\rm I}^\Omega$, and replace the last row of $W_{\rm II}$ and $B_{\rm II}$ in Eq. \eqref{eq:W2B2}, i.e. $O$ and $R$, by $w_{\rm II}^\Omega$ and $b_{\rm II}^\Omega$. Finally, we can check that the modified
\begin{equation}
f(x):=w_{\rm III}\sigma(W_{\rm II}\sigma(W_{\rm I}x+B_{\rm I})+B_{\rm II})=-\phi^\Omega(x)+\sum_{i=1}^{N_{\mathcal{T}}}\phi^{\tau_i}(x)
\end{equation}
satisfies:
\begin{itemize}
\item $f|_{\Omega_\T^\epsilon}=v|_{\Omega_\T^\epsilon}$.
\item $\norm{f}_{L^\infty(\Omega)}\leq 2R$.
\item $f|_{\R^n\backslash\Omega}=0$.
\end{itemize}
It is the compactly supported neural network we need.

At last we show that two hidden layers are necessary when $n\geq 2$. Since the set of all the non-differentiable points of a nonzero single-hidden-layer NN is the union of finitely many hyperplanes, we assert that such a NN does not have a compact support. Otherwise, the non-differentiable points are bounded, which contradicts the unboundedness of hyperplanes in $n\geq 2$ dimensional space. Next we consider the subdomains where the NN acts as a linear function. As the NN is not compactly supported, there exists at least one nonzero unbounded linear subdomain, on which the NN's $L^p$ norm is infinity. It tells that the single-hidden-layer NN is not a $L^p(\R^n)$ function.
\end{proof}
In fact single-hidden-layer NNs are sufficient to approximate $L^p(\R)$ functions in one dimensional space, due to the boundedness of one dimensional hyperplanes, which are exactly points.

The above results establish a bridge between two-hidden-layer ReLU FNNs and constant/linear finite element functions. The analysis of approximation capability of FNNs can be guided by the estimation of finite element functions via
\begin{equation}
\inf_{f\in{\rm FNN}(h_1,h_2)}\norm{f-v}_{L^p(\Omega)}\leq\inf_{u\in{\rm FE}(h_1,h_2)}\norm{u-v}_{L^p(\Omega)},\quad\forall v\in L^p(\Omega),
\end{equation}
where
\begin{equation}
{\rm FE}(h_1,h_2):=\bigcup_{2H_\T^i+H_\T^b\leq h_1,N_\T+1\leq h_2}\{{\rm finite\ element\ functions\ on\ mesh\ \T}\}.
\end{equation}
For instance, we derive an error estimate based on the standard simplex mesh. Below $W^{2,p}$ denotes the Sobolev space and $|\cdot|_{W^{2,p}}$ denotes the semi-norm.
\begin{corollary}
    Suppose that $v\in W^{2,p}(\Omega)$ with $\Omega=[0,1]^n$ and $1\leq p<\infty$, $N$ is a positive integer, then we have the following estimate
    \begin{equation}
        \inf_{f\in{\rm FNN}(2n^2N-n^2+n,N^n\cdot n!+1)}\norm{f-v}_{L^p(\Omega)}\leq |v|_{W^{2,p}(\Omega)}\cdot \mathcal{O}(N^{-2}).
    \end{equation}
\end{corollary}
\begin{proof}
Let $\T$ be the standard simplex mesh of $\Omega$, i.e. the Freudenthal triangulation \cite{freudenthal1942simplizialzerlegungen}. Assume that $\T$ is of size $h=\frac{1}{N}$, which means there are $(N+1)^n$ vertexes, and
\begin{equation}
    H_\T^i=n^2N-\frac{n(n+1)}{2},\quad H_\T^b=2n,\quad N_\T=N^n\cdot n!.
\end{equation}
Denote by $V_h$ the set of linear finite element functions on $\T$. The classical error estimate in finite element methods \cite{brenner2008mathematical,ciarlet2002finite} shows that
\begin{equation}
    \inf_{f\in V_h}\norm{f-v}_{L^p(\Omega)}\leq|v|_{W^{2,p}(\Omega)}\cdot \mathcal{O}(N^{-2}),
\end{equation}
consequently
\begin{equation}
\begin{split}
&\inf_{f\in{\rm FNN}(2n^2N-n^2+n,N^n\cdot n!+1)}\norm{f-v}_{L^p(\Omega)}\\=&\inf_{f\in{\rm FNN}(2H_\T^i+H_\T^b,N_\T+1)}\norm{f-v}_{L^p(\Omega)}\\\leq&\inf_{f\in V_h}\norm{f-v}_{L^p(\Omega)}\\\leq&|v|_{W^{2,p}(\Omega)}\cdot \mathcal{O}(N^{-2}).
\end{split}
\end{equation}
That is the estimate we need.
\end{proof}
\begin{remark}
Note that the number of nonzero parameters in $f$ is $\mathcal{O}(N^n)$ according to the construction of weak representation, as the weight matrix in the second layer is in fact sparse. Therefore the approximated $f$ has the same scale of free parameters as the degree of freedom on mesh $\T$. 
\end{remark}
Further exploration is expected in the future.

\subsection{Example}
We show two simple examples for the application of the theory. The first example is a 2-d convex polygon mesh for constant finite element functions shown in Figure \ref{fig:constant}. We observe that in this case there are 24 interior lines and 5 boundary lines, so that the number of neurons of the first layer is $2\times 24+5=53$. Moreover, there are 18 polygons in this mesh, hence the number of neurons of the second layer is $18+1=19$. Any constant finite element function on this mesh can be weakly represented by an FNN of size [2-53-19-1].

\begin{figure}[htbp]
\centering
\includegraphics[width=0.9\textwidth]{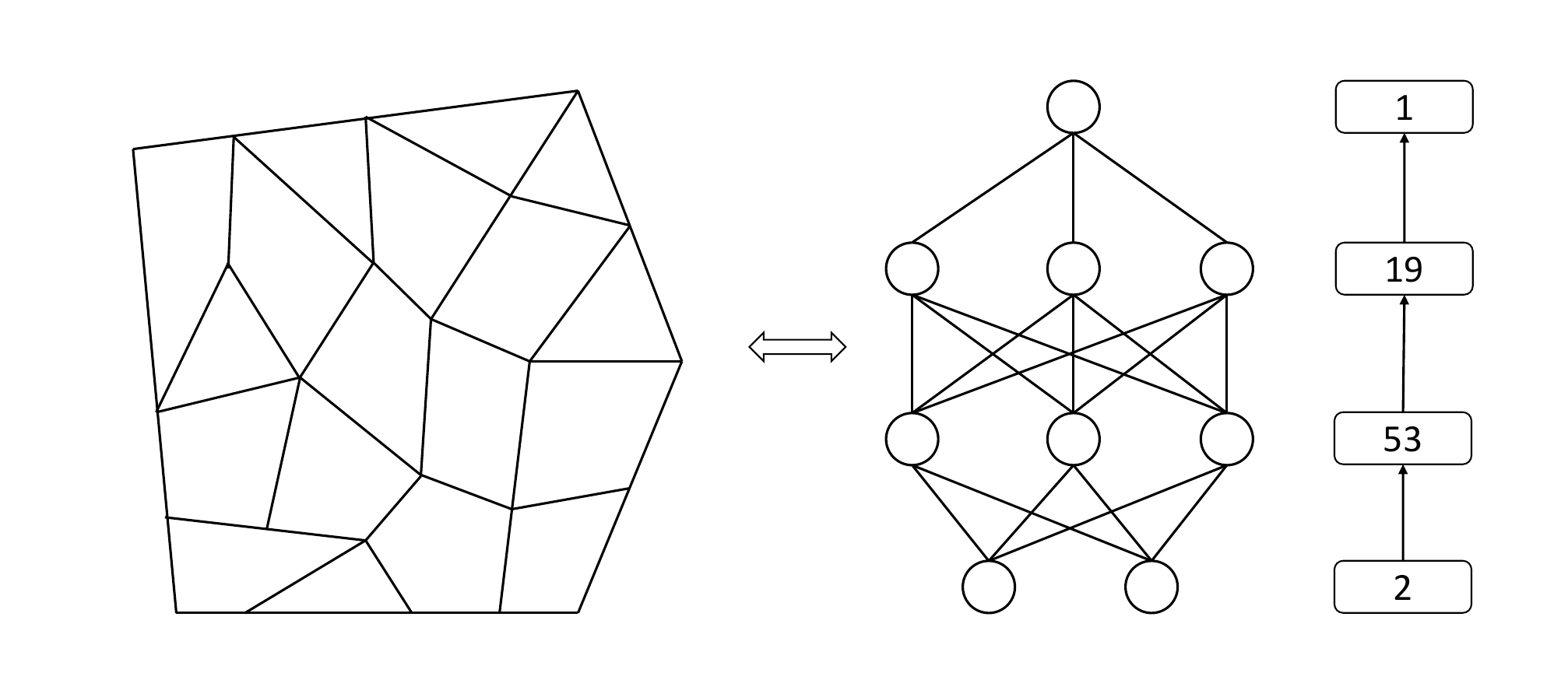}
\caption{A 2-d convex polygon mesh for constant finite element functions and its corresponding FNN size for weak representation.}
\label{fig:constant}
\end{figure}

The second example is a 2-d simplex mesh for linear finite element functions shown in Figure \ref{fig:linear}. In this case there are 13 interior lines and 4 boundary lines, so that the number of neurons of the first layer is $2\times 13+4=30$. Moreover, there are 32 simplexes in this mesh, hence the number of neurons of the second layer is $32+1=33$. Any linear finite element function on this mesh can be weakly represented by an FNN of size [2-30-33-1].

\begin{figure}[htbp]
\centering
\includegraphics[width=0.9\textwidth]{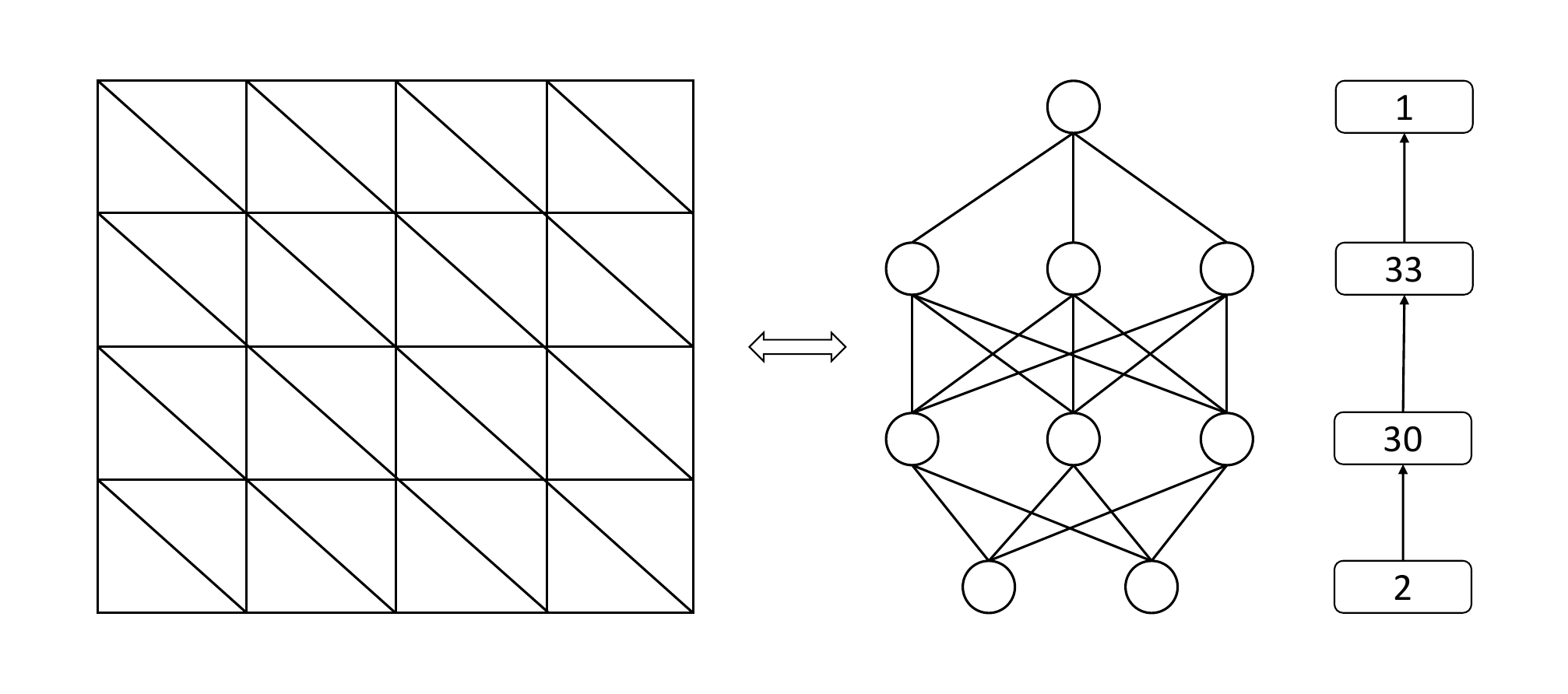}
\caption{A 2-d simplex mesh for linear finite element functions and its corresponding FNN size for weak representation.}
\label{fig:linear}
\end{figure}

\section{Tensor finite element}\label{sec:tensor}
The case of tensor finite element is quite easy to clarify compared to constant/linear finite element. Consider the product domain
\begin{equation}
\Omega=\Omega_1\times\Omega_2\times\cdots\times\Omega_n,
\end{equation}
where $\Omega_i=[a_i,b_i]$ is a 1-d interval, $1\leq i\leq n$. The tensor-type mesh of $\Omega$ can be written as
\begin{equation}
\T=\{\tau_{i_1,i_2,...,i_n}\}_{1\leq i_k\leq N_k},\quad \tau_{i_1,i_2,...,i_n}=[t_{i_1-1}^1,t_{i_1}^1]\times[t_{i_2-1}^2,t_{i_2}^2]\times\cdots\times[t_{i_n-1}^n,t_{i_n}^n],
\end{equation}
where
\begin{equation}
a_i=t_{0}^i<t_{1}^i<\cdots<t_{N_i}^i=b_i,\quad 1\leq i\leq n.
\end{equation}

Consider the tensor finite element functions on $\T$ which are continuous piecewise multilinear functions defined as
\begin{equation}
t_p(x_1,x_2,...,x_n):=\sum_{\delta=(\delta_1,\delta_2,...,\delta_n)\in\{0,1\}^n}c_\delta^\tau x_1^{\delta_1}x_2^{\delta_2}\cdots x_n^{\delta_n}\quad {\rm for}\quad (x_1,x_2,...,x_n)\in\tau,\tau\in\T,
\end{equation}
where $c_\delta^\tau\in\R$ is related to $\tau$, and $t_p$ is continuous over $\Omega$. Denote the set of all the multilinear tensor finite element functions on $\T$ by $\mathcal{U}_{\T}$. Note that $\mathcal{U}_{\T}$ is a linear space of dimension $\Pi_{i=1}^n(N_i+1)$, which equals the number of the vertexes in this mesh \cite{brenner2008mathematical}.

\subsection{Strict representation}
Different from the previous cases, tensor neural networks are easy to strictly represent the tensor finite element functions with clear size. Denote by $r$ the maximum rank of the $n$-order tensor with dimensions of $(N_1+1)\times(N_2+1)\times\cdots\times(N_n+1)$.
\begin{theorem}\label{thm:tnn}
${\rm TNN}^{r}(N_1+1,N_2+1,...,N_n+1)$ strictly represents $\mathcal{U}_{\T}$, i.e.,
\begin{equation}
{\rm TNN}^{r}(N_1+1,N_2+1,...,N_n+1)\supset\mathcal{U}_{\T}.
\end{equation}
\end{theorem}
\begin{proof}
Denote by $\phi_{i_1,i_2,...,i_n}$ the nodal function corresponding to $(t_{i_1}^1,t_{i_2}^2,...,t_{i_n}^n)$, $0\leq i_k\leq N_k$. Here $\phi_{i_1,i_2,...,i_n}(t_{i_1}^1,t_{i_2}^2,...,t_{i_n}^n)=1$ and $\phi_{i_1,i_2,...,i_n}(t_{j_1}^1,t_{j_2}^2,...,t_{j_n}^n)=0$ for $(i_1,i_2,...,i_n)\neq(j_1,j_2,...,j_n)$. Given any $u\in\mathcal{U}_{\T}$, it can be expressed as
\begin{equation}
u=\sum_{0\leq i_k\leq N_k,1\leq k\leq n}c_{i_1,i_2,...,i_n}\phi_{i_1,i_2,...,i_n}.
\end{equation}
We further consider the nodal functions for 1-d intervals. Denote by $\phi_i^k$ the 1-d nodal function corresponding to $t_i^k$ for interval $[a_k,b_k]$, $0\leq i\leq N_k$, $1\leq k\leq n$. Here $\phi_i^k(t_i^k)=1$ and $\phi_i^k(t_j^k)=0$ for $i\neq j$. Then we have
\begin{equation}
\phi_{i_1,i_2,...,i_n}=\phi_{i_1}^1\otimes\phi_{i_2}^2\otimes\cdots\otimes\phi_{i_n}^n,
\end{equation}
thus
\begin{equation}
u=\sum_{0\leq i_k\leq N_k,1\leq k\leq n}c_{i_1,i_2,...,i_n}\phi_{i_1}^1\otimes\phi_{i_2}^2\otimes\cdots\otimes\phi_{i_n}^n.
\end{equation}
Since
\begin{equation}
(c_{i_1,i_2,...,i_n})_{(N_1+1)\times(N_2+1)\times\cdots\times(N_n+1)}
\end{equation}
is a tensor of order $n$ with shape $(N_1+1)\times(N_2+1)\times\cdots\times(N_n+1)$, by the definition of $r$, we can find $c_{k,i_k}^p\in\R$, $0\leq i_k\leq N_k$, $1\leq k \leq n$, $1\leq p \leq r$, such that
\begin{equation}
c_{i_1,i_2,...,i_n}=\sum_{p=1}^rc_{1,i_1}^pc_{2,i_2}^p\cdots c_{n,i_n}^p,\quad 0\leq i_k\leq N_k,1\leq k\leq n.
\end{equation}
Subsequently, we have
\begin{equation}
\begin{split}
u=&\sum_{p=1}^r\sum_{0\leq i_k\leq N_k,1\leq k\leq n}c_{1,i_1}^pc_{2,i_2}^p\cdots c_{n,i_n}^p\phi_{i_1}^1\otimes\phi_{i_2}^2\otimes\cdots\otimes\phi_{i_n}^n\\
=&\sum_{p=1}^r\left(\sum_{i_1=0}^{N_1}c_{1,i_1}^p\phi_{i_1}^1\right)\otimes\left(\sum_{i_2=0}^{N_2}c_{2,i_2}^p\phi_{i_2}^2\right)\otimes\cdots\otimes\left(\sum_{i_n=0}^{N_n}c_{n,i_n}^p\phi_{i_n}^n\right).
\end{split}
\end{equation}
Now consider the $f^p_k:=\sum_{i_k=0}^{N_k}c_{k,i_k}^p\phi_{i_k}^k$, which is exactly a continuous piecewise linear function on $\Omega_k$ with nodes $t^k_i$, and satisfies $f_k^p(t_i^k)=c^p_{k,i}$, $0\leq i\leq N_k$. Denote
\begin{equation}
W_k:=(1,1,...,1,0)^T\in\R^{N_k+1},\quad b_k:=(-t_0^k,-t_1^k,...,-t_{N_k-1}^k, 1)^T\in\R^{N_k+1},
\end{equation}
and
\begin{equation}
w_k^p:=(\mu_0^p,\mu_1^p,...,\mu_{N_k}^p),
\end{equation}
which is the solution of the linear system
\begin{equation}
\begin{pmatrix}
0 & & & & & 1 \\
t_1^k-t_0^k & 0 & & & & 1\\
t_2^k-t_0^k & t_2^k-t_1^k & 0 & & & 1\\
\vdots & \vdots & \ddots & \ddots & & \vdots\\
t_{N_k-1}^k-t_0^k & t_{N_k-1}^k-t_1^k & \cdots & t_{N_k-1}^k-t_{N_k-2}^k & 0 & 1\\
t_{N_k}^k-t_0^k & t_{N_k}^k-t_1^k & \cdots & t_{N_k}^k-t_{N_k-2}^k& t_{N_k}^k-t_{N_k-1}^k & 1\\
\end{pmatrix}\begin{pmatrix}
\mu_0^p\\
\mu_1^p \\
\vdots \\
\mu_{N_k}^p
\end{pmatrix}=\begin{pmatrix}
c_{k,0}^p \\
c_{k,1}^p \\
\vdots \\
c_{k,N_k}^p \\
\end{pmatrix}.
\end{equation}
One can check that $l^p_k(x):=w_k^p\sigma(W_kx+b_k)=f_k^p(x)$ for $x\in[a_k,b_k]$. Therefore
\begin{equation}
u=\sum_{p=1}^rl_1^p\otimes l_2^p\otimes\cdots\otimes l_n^p\in{\rm TNN}^{r}(N_1+1,N_2+1,...,N_n+1),
\end{equation}
and the proof is finished.
\end{proof}
Although we provide a definite $r$ for the size of TNN, it is in fact not computable for high-dimensional case, since how to determine the maximum rank of high-order tensors is still a difficult unresolved problem \cite{haastad1989tensor}. With Theorem \ref{thm:tnn}, we can also study the approximation capability of TNNs by investigating tensor finite element functions.

\subsection{Example}
We show an example of the strict representation for tensor finite element functions via TNNs. The mesh is presented in Figure \ref{fig:tensor}. Since it is a 4-by-5 rectangular mesh, and the maximum rank of the 4-by-5 matrix is 4, we know $r=4$. In addition, the numbers of neurons of the two branch nets are 4 and 5, respectively. Therefore the corresponding size is [1-4-4] for the first branch net and [1-5-4] for the second branch net.  

\begin{figure}[htbp]
\centering
\includegraphics[width=1.0\textwidth]{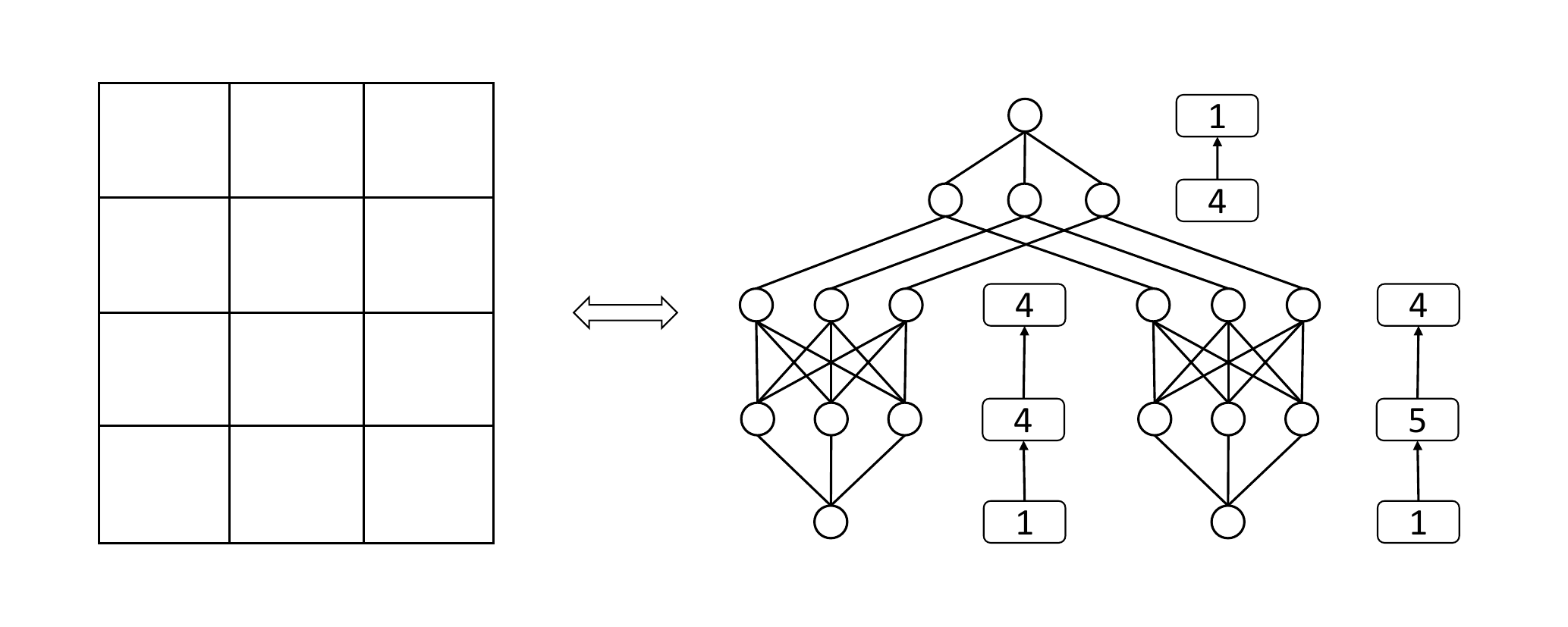}
\caption{A 2-d tensor-type mesh for continuous piecewise bilinear functions and its corresponding TNN size for strict representation.}
\label{fig:tensor}
\end{figure}

\section{Conclusions}\label{sec:conclusions}
We discussed the relationship between the ReLU NNs and the finite element functions. The main content has two parts: The first part focuses on the two-hidden-layer ReLU NNs and the constant/linear finite element functions. We gave the concept of weak representation and proved that piecewise linear functions on a convex polytope mesh can be weakly represented by two-hidden-layer ReLU NNs. In addition, the numbers of neurons of the two hidden layers required to weakly represent were accurately given based on the numbers of polytopes and hyperplanes involved in this mesh. Such weak representation leads to a perspective for analyzing approximation capability of ReLU NNs in $L^p$ norm via finite element functions. The second part shows that the recent tensor neural networks can strictly represent the tensor finite element functions. Furthermore, for constant, linear and tensor finite element functions, several specific examples were presented and demonstrated how to compute the numbers of neurons for representation given meshes.

\bibliographystyle{abbrv}
\bibliography{references}

\end{document}